\documentclass[11pt]{article}
\usepackage{amsthm,amsmath,amssymb,mathabx}
\usepackage[numbers]{natbib}
\usepackage{times}
\usepackage{indentfirst}
\usepackage{calrsfs}
\usepackage{color}

\def\R{{\mathbb R}}
\def\P{{\mathbb P}}
\def\E{{\mathbb E}}
\def\I{{\mathbb I}}

\newtheorem{theorem}{Theorem}

\theoremstyle{definition}
\newtheorem{definition}[theorem]{Definition}

\begin{document}

\title{On extremes of a random walk with positive drift\\
over an intermediate regularly varying time interval}

\author{Sergey Foss\thanks{School of MACS, Heriot-Watt University, Edinburgh, UK, 
s.foss@hw.ac.uk}\ \
and Dmitry Korshunov\thanks{Lancaster University, UK, d.korshunov@lancaster.ac.uk}}

\maketitle

\begin{abstract}
We consider a random walk $\{S_n\}$ with a finite positive drift that is stopped at a random time $\tau$ 
having an intermediate regularly varying distribution. We assume that the jump distribution is lighter-tailed 
than the distribution of $\tau$. Under these conditions, we show that the tails of the distributions of 
$S_{\tau}$ and $M_{\tau} = \max_{k\le \tau} S_k$ are asymptotically equivalent and are determined 
by the tail of $\tau$, while the random walk $\{S_n\}$ contributes only through the law of large numbers.
\end{abstract}



\section{Introduction}
\label{sec:intro}

Let $\xi$, $\xi_1$, $\xi_2$, \ldots\ be independent and identically distributed random variables 
with finite positive mean $a:=\E\xi>0$. Let $S_0=0$ and $S_n=\xi_1+\ldots+\xi_n$, $n=1$, $2$, \ldots, 
be the random walk with positive drift generated by $\{\xi_n\}$, and let $M_n=\max_{0\le i\le n}S_i$ 
denote its running maximum.

We say that a random variable $\xi$ has a {\it light-tailed} distribution if $\E e^{\lambda\xi}<\infty$ 
for some $\lambda>0$. Conversely, we say that $\xi$ has a {\it heavy-tailed} distribution if 
$\E e^{\lambda\xi}=\infty$ for all $\lambda>0$.

Let $\tau$ be a counting random variable. In various contexts, we are interested in the distribution 
of the random walk $S_n$ observed at the random time $\tau$, or in the distribution of its maximum 
over the random time interval $[0,\tau]$. It is known that the tail asymptotics of $S_\tau$ and $M_\tau$ 
depend crucially on the drift of $S_n$. In the case of negative drift, the strong law of large numbers implies 
that both $S_\tau$ and $M_\tau$ are bounded above by the proper random variable $M_\infty$. 
In this case, the tail asymptotics of $S_\tau$ and $M_\tau$ are determined primarily by the tail 
of the distribution of $\xi$. In particular, the distribution of $M_\infty$ is light-tailed if and only
 if the distribution of $\xi$ is light-tailed. For a recent paper on heavy-tailed case, see, e.g. \cite{FKP}.

In the case of positive drift, $M_\infty=\infty$ a.s. by the strong law of large numbers, 
and the tail asymptotics of both $S_\tau$ and $M_\tau$ may be determined by either the tail 
of $\tau$ or that of $\xi$, or by both. Thus, there are three main cases, 
depending on whether the tail distribution of $\tau$ is lighter than that of $\xi$, 
heavier than that of $\xi$, or tail-comparable to it; see, e.g., Stam \cite{Stam} or Denisov et al. \cite{DFK}.

In this paper, we consider the case of positive drift. Our goal is to study conditions on $\{\xi_n\}$ and $\tau$ 
under which the tail probabilities $\P\{S_\tau>x\}$ and $\P\{M_\tau>x\}$, as $x\to\infty$, 
are determined by the tail distribution of $\tau$, while the contribution of $\{\xi_n\}$ 
is only at the level of the law of large numbers. The expected asymptotic relation is
\begin{eqnarray}\label{asy.eq}
\P\{S_\tau>x\} &\sim& \P\{\tau>x/\E\xi\}\quad\mbox{as }x\to\infty.
\end{eqnarray}
Hereinafter, for any two eventually positive functions $f(x)$ and $g(x)$, 
we write $f(x)\sim g(x)$ if $f(x)/g(x)\to 1$ as $x\to\infty$. 
Note that the relation \eqref{asy.eq} can hold only if the tail distribution of $\tau$ is heavier than that of $\xi$.
 
This problem was introduced by Schmidli \cite{Schmidli} in the case where
$\{\xi_n\}$ and $\tau$ are independent and $\tau$ has a subexponential distribution.

Later, this problem was analysed by Robert and Segers in \cite{RS} under the assumption
that $\tau$ has intermediate regularly varying distribution.
In Theorem 3.1 there, they have proved \eqref{asy.eq} under arbitrary dependence between
$\{\xi_n\}$ and $\tau$, provided that $\xi_1\ge 0$,
$\E\xi_1^r<\infty$ for some $r>1$, and $x\P\{\xi_1>x\}=o(\P\{\tau>x\})$.
In Theorem 3.2, they considered the case where $\{\xi_n\}$ and $\tau$
are independent, $\xi_1\ge 0$, $\E\xi_1^r<\infty$ for some $r>1$, 
and $\E(\tau\wedge x)=O(x^q\P\{\tau>x\})$ for some $q\le r$.

Ale\v skevi\v cien\.e et al. proved \eqref{asy.eq} in Theorem 1.2 of \cite{ALS}
under the assumptions of independence of $\{\xi_n\}$ and $\tau$, $\xi_1\ge 0$, 
$\E\tau<\infty$, intermediate regular variation of $\tau$, and $\P\{\xi_1>x\}=o(\P\{\tau>x\})$.

We aim to understand the best possible relations between $\xi$ and $\tau$ 
under various assumptions on the dependence between the random variables $\{\xi_n\}$ and $\tau$.
To begin with, we recall the following notion of independence of the future
(first considered by Kolmogorov and Prokhorov in \cite{KP}),
which generalises both the notions of stopping times and standard independence.

\begin{definition}\label{def:dndf}
We say that a counting random time $\tau$ {\it does not depend on, 
or is independent of, the future} 
of the sequence $\{\xi_n\}$ if the sigma-algebras
$\sigma\{\xi_i,\ i\le n,\ \I\{\tau\le n\}\}$
and $\sigma\{\xi_i,\ i>n\}$ are independent for all $n\ge 1$. 
\end{definition}
Recall that a counting random variable $\tau$ is a stopping time (w.r.t. the sequence $\{\xi_n\}$ if, for any $n\ge 1$, 
the event $\{\tau \le n\}$ belongs to the sigma-algebra  $\sigma\{\xi_i,\ i\le n\}$. 
Notice that any stopping time is independent of the future of the sequence $\{\xi_n\}$, 
but the converse is not true.

If we introduce a filtration $\mathcal F_n$ by setting
$\mathcal F_n=\sigma\{\xi_i,\ \I\{\tau=i\},\ i\le n\}$, then 
Definition \ref{def:dndf} may be reformulated equivalently as saying that
$\mathcal F_n$ is independent of $\sigma\{\xi_i,\ i>n\}$ for all $n$.
Notice that this specific filtration is the minimal one that satisfies this independence property; 
compare to Kulik and Soulier \cite[Sec. 4.1]{KS}.

Next we recall the class of distributions that plays the main role in our analysis.

\begin{definition}\label{def:irvd}
We say that a distribution $G$ is
{\it intermediate regularly varying} at infinity (a class introduced by Cline in \cite{Cline}) if
\begin{eqnarray}\label{i.r.v}
\lim_{\varepsilon\downarrow 0}\limsup_{x\to\infty}
\frac{\overline G((1-\varepsilon)x)}{\overline G(x)}
&=& 1.
\end{eqnarray}
Equivalently (see Theorem 2.47 in \cite{FKZ}), $G$ is 
an intermediate regularly varying distribution if $\overline G(x) \sim \overline G(x+h(x))$
for any function $h$ such that $h(x) = o(x)$, i.e. $h(x)/x \to 0$ as $x\to\infty$.
\end{definition}

In particular, the condition \eqref{i.r.v} holds true for any 
{\it regularly varying at infinity} distribution $G$ with index $\beta\ge 0$, that is, if
$\overline G(cx)\sim\overline G(x)/c^\beta$ as $x\to\infty$, whatever $c>0$.

It follows from Definition \ref{def:irvd} that, for any $\gamma>0$, 
there exists $c(\gamma)<\infty$ such that
\begin{eqnarray}\label{i.r.v.gamma}
\overline G(\gamma x) &\le& c(\gamma)\overline G(x)\quad\mbox{for all }x>0.
\end{eqnarray}

In Section \ref{sec:ind.f}, we prove \eqref{asy.eq} under minimal conditions
in the case where $\tau$ is independent of the future, while in Section \ref{sec:arb}
we consider an arbitrary dependent sequence $\{\xi_n\}$ and random time $\tau$.

In Section \ref{sec:cycle}, we derive some tail results for the length
of the cycle of a random walk with negative drift, which are 
needed in the proof in Section \ref{sec:ind.f}.

\section{The case of $\tau$ being independent of the future}
\label{sec:ind.f}

The aim of our first main result is fourfold.
We do not assume that $\xi\ge 0$ and require only the finiteness of the first moment of $\xi$, in order to emphasize the role of the law of large numbers.
We consider a very general dependence structure between $\{\xi_n\}$ and $\tau$.
We also assume the weakest possible asymptotic relation between the distribution tails of $\xi$ and $\tau$.
Taken together, these features necessitate rather sophisticated proof techniques.

For comparison, the proof of a similar theorem (Theorem 3.2) in \cite{RS} is relatively short,
mainly due to the independence assumption between ${\xi_n}$ and $\tau$,
as well as the additional conditions $\xi_1\ge 0$ and $\E\xi_1^r<\infty$ for some $r>1$.
These assumptions allow the authors of \cite{RS} to apply results already available in the literature under the same conditions.
Let us note that, apart from the initial step, the arguments used there do not
apply under the conditions considered below; therefore, a different proof technique is required.

\begin{theorem}\label{thm:asy}
Assume that $a:=\E\xi>0$.
Let the distribution of $\tau$ be intermediate regularly varying at infinity and such that
\begin{eqnarray}\label{F.tau.cond}
\P\{|\xi|>x\} &=& o\biggl(\frac{\P\{\tau>x\}}{\E(\tau\wedge x)}\biggr)\quad\mbox{as }x\to\infty.
\end{eqnarray}
If $\tau$ does not depend on the future, then
\begin{eqnarray}\label{F.tau.asy}
\P\{M_\tau>x\}\ \sim\ \P\{S_\tau>x\} &\sim& \P\{\tau>x/a\}\quad\mbox{as  } x\to\infty.
\end{eqnarray}
\end{theorem}

We can specify the condition \eqref{F.tau.cond} in different cases as follows.
\begin{itemize}
\item[(i)] In the case where $\E\tau<\infty$, \eqref{F.tau.cond} is equivalent to
\begin{eqnarray}\label{F.tau.cond.finite}
\P\{|\xi|>x\} &=& o(\P\{\tau>x\})\quad\mbox{as }x\to\infty.
\end{eqnarray}
\item[(ii)]  In the case of regularly varying at infinity distribution of $\tau$
with index $\beta\in[0,1)$, so that $\E\tau=\infty$, by Karamata's theorem,
\begin{eqnarray*}
\E(\tau\wedge x) &=& \int_0^x\P\{\tau>y\}dy
\ \sim\ \frac{1}{1-\beta}x\P\{\tau>x\}\quad\mbox{as }x\to\infty.
\end{eqnarray*}
Hence, in this case, \eqref{F.tau.cond} reduces to
\begin{eqnarray*}
\P\{|\xi|>x\} &=& o(1/x)\quad\mbox{as }x\to\infty,
\end{eqnarray*}
which is satisfied automatically due to $\E|\xi|<\infty$.
\item[(iii)]  In the case where $\P\{\tau>x\}\sim 1/x\log^\beta x$
with $\beta\in(-\infty,1)$, we have
\begin{eqnarray*}
\E(\tau\wedge x) &\sim& \frac{\log^{1-\beta} x}{1-\beta}
\ \sim\ \frac{1}{1-\beta} x\log x \P\{\tau>x\}\quad\mbox{as }x\to\infty.
\end{eqnarray*}
Therefore, in this case, \eqref{F.tau.cond} is equivalent to the condition
\begin{eqnarray*}
\P\{|\xi|>x\} &=& o(1/x\log x)
\ =\ o(\P\{\tau>x\}\log^{1-\beta} x)\ \mbox{as }x\to\infty.
\end{eqnarray*}
\item[(iv)] More generally, let, for some $k\ge 1$ and $\beta\in(-\infty,1)$,
\begin{eqnarray*}
\P\{\tau>x\} &\sim& \frac{1}{x\log x\cdot\ldots\cdot\log_{(k-1)}x
\cdot\log_{(k)}^\beta x}\quad\mbox{as }x\to\infty,
\end{eqnarray*}
where $\log_{(i)} x$ stands for the $i$th iterated logarithm function. Then
\begin{eqnarray*}
\E(\tau\wedge x) &\sim& \frac{\log_{(k)}^{1-\beta} x}{1-\beta}\\
&\sim& \frac{1}{1-\beta} x\log x\cdot\ldots\cdot\log_{(k)} x 
\P\{\tau>x\}\quad\mbox{as }x\to\infty.
\end{eqnarray*}
Therefore, in this case, \eqref{F.tau.cond} is equivalent to the condition
\begin{eqnarray}\label{F.tau.cond.inf.log}
\P\{|\xi|>x\} &=& o(1/x\log x\cdot\ldots\cdot\log_{(k)} x)\nonumber\\
&=& o(\P\{\tau>x\}\log_{(k)}^{1-\beta} x)\quad\mbox{as }x\to\infty.
\end{eqnarray}
The larger the value of $k$ we consider, the closer $\tau$ to the integrable case (i),
and the closer the condition \eqref{F.tau.cond.inf.log} to the condition \eqref{F.tau.cond.finite}.
\end{itemize}

\begin{proof}[Proof of Theorem \ref{thm:asy}]
We will obtain two asymptotically matching bounds:
an upper bound for $\P\{M_\tau >x\}$
and a lower bound for $\P\{S_\tau>x\}$.
Since, for all $x$, the former probability is greater than the latter,
this completes the proof.

We start with an upper tail bound for the maximum.
Fix $\varepsilon>\delta>0$. Then
\begin{eqnarray}\label{upper.1}
\P\{M_\tau>x\} &\le& \P\Bigl\{M_\tau>x,\ \tau\le\frac{x}{a+\varepsilon}\Bigr\}
+\P\Bigl\{\tau>\frac{x}{a+\varepsilon}\Bigr\}.
\end{eqnarray}
Let us bound the first probability on the right hand side as follows:
\begin{eqnarray*}
 \P\Bigl\{M_\tau>x,\ \tau\le\frac{x}{a+\varepsilon}\Bigr\} &=&
\P\Bigl\{M_\tau-(a+\delta)\tau>x-(a+\delta)\tau,\ \tau\le\frac{x}{a+\varepsilon}\Bigr\}\\
&\le& \P\Bigl\{M_\tau-(a+\delta)\tau>x\frac{\varepsilon-\delta}{a+\varepsilon},\ 
\tau\le\frac{x}{a+\varepsilon}\Bigr\} =: P_1(x).
\end{eqnarray*}

Consider an auxiliary random walk 
$\widehat{S}_n:=S_n-(a+\delta)n 
= \sum_{i=1}^n (\xi_i - (a+\delta)) =: \sum_{i=1}^n \widehat{\xi}_i$ 
with negative drift ${\mathbb E} \widehat{\xi_1} =-\delta$,  
and let $\widehat{M}_n:=\max_{0\le k\le n}\widehat{S}_k$ be its running maximum.  
Then
\begin{eqnarray*}
P_1(x) &\le& 
\P\Bigl\{\widehat{M}_\tau>x\frac{\varepsilon-\delta}{a+\varepsilon},\ 
\tau\le\frac{x}{a+\varepsilon}\Bigr\}
\ \le\ \P\Bigl\{
\widehat{M}_{\tau(x)}>x\frac{\varepsilon-\delta}{a+\varepsilon}\Bigr\},
\end{eqnarray*}
where $\tau(x):=\tau\wedge\frac{x}{a+\varepsilon}$. If
\begin{eqnarray}\label{Mtau.upper}
\P\Bigl\{\widehat{M}_{\tau(x)}>x\frac{\varepsilon-\delta}{a+\varepsilon}\Bigr\}
&=& o(\P\{\tau>x\})\quad\mbox{as }x\to\infty,
\end{eqnarray}
then by substituting this upper bound into \eqref{upper.1} 
and taking into account the upper bound \eqref{i.r.v.gamma},
we derive that, for any fixed $\varepsilon>0$,
\begin{eqnarray*}
\P\{M_\tau>x\} &\le& 
(1+o(1))\P\Bigl\{\tau>\frac{x}{a+\varepsilon}\Bigr\}
\quad\mbox{as }x\to\infty.
\end{eqnarray*}
Since the distribution of $\tau$ is assumed to be intermediate regularly varying
and $\varepsilon$ can be chosen arbitrarily close to zero,
we obtain the following upper bound: 
\begin{eqnarray}\label{S.tau.upper.bound}
\P\{M_\tau>x\} &\le& (1+o(1))\P\{\tau>x/a\}\quad\mbox{as }x\to\infty.
\end{eqnarray}

So, it remains to show \eqref{Mtau.upper}.
Let $\theta_0=0$. By the strong law of large numbers, 
the stopping time for a negatively drifted random walk,
\begin{eqnarray*}
\theta_1 &:=& \inf\{n\ge 1:\ \widehat S_n>0\},
\end{eqnarray*}
is an improper random variable, that is, $p:=\P\{\theta_1<\infty\}<1$.
For $k\ge 2$, we define recursively stopping times $\theta_k$ as follows:
if $\theta_{k-1}<\infty$, then
\begin{eqnarray*}
\theta_k &:=& \inf\{n>\theta_{k-1}:\ \widehat S_n-\widehat S_{\theta_{k-1}}>0\},
\end{eqnarray*}
and if $\theta_{k-1}=\infty$, then $\theta_k=\infty$. 
Then, by the renewal structure of the process,
\begin{eqnarray*}
\P\{\theta_k<\infty\mid\theta_{k-1}<\infty\} &=& \P\{\theta_1<\infty\}\ =\ p\ <\ 1,
\end{eqnarray*}
and, therefore, $\P\{\theta_k<\infty\}=p^k$ for all $k\ge 1$.
Note that given  $\theta_k<\infty$, 
we have that $\widehat S_n-\widehat S_{\theta_{j-1}}\le 0$ for all 
$\theta_{j-1}\le n<\theta_j$ and all $j\le k$. Then, for $y>0$,
\begin{eqnarray}\label{upper.tau.crp.2.sum}
\lefteqn{\P\{\theta_k\le\tau(x),\ \theta_{k+1}=\infty,\ \widehat M_{\theta_k}>y\}}\nonumber\\
&=& \P\{\widehat S_{\theta_k}>y,\ \theta_k\le\tau(x),\ \theta_{k+1}=\infty\}\nonumber\\
&\le& \P\Bigl\{\theta_k\le\tau(x),\ \theta_{k+1}=\infty,\ \bigcup_{j=1}^k
\{\widehat S_{\theta_j}-\widehat S_{\theta_{j-1}}>y_{k,j}\}\Bigr\}\nonumber\\
&\le& \sum_{j=1}^k
\P\{\theta_k\le\tau(x),\ \theta_{k+1}=\infty,\ 
\widehat S_{\theta_j}-\widehat S_{\theta_{j-1}}>y_{k,j}\},
\end{eqnarray}
where  $\{y_{k,j}\}$ is any sequence of non-negative numbers such that 
$y_{k,1}+\ldots+y_{k,k}\le y$. We have, for all $j\le k$,
\begin{eqnarray}\label{S.theta.1-}
\lefteqn{\P\{\theta_k\le\tau(x),\ \theta_{k+1}=\infty,\ 
\widehat S_{\theta_j}-\widehat S_{\theta_{j-1}}>y_{k,j}\}}\nonumber\\
&\le& \P\{\theta_j\le\tau(x),\ \theta_{j+1},\ldots,\theta_k<\infty,\ \theta_{k+1}=\infty,\ 
\widehat S_{\theta_j}-\widehat S_{\theta_{j-1}}>y_{k,j}\}\nonumber\\
&\le& \P\{\theta_j\le\tau(x),\ \widehat\xi_{\theta_j}>y_{k,j}\}p^{k-j}(1-p),
\end{eqnarray}
because the event $\{\tau(x)\ge\theta_j,\ \widehat\xi_{\theta_j}>y_{k,j}\} = 
\{\overline{\tau(x)<\theta_j},\ \widehat\xi_{\theta_j}>y_{k,j}\}$
does not depend on the future increments
$\{\widehat{\xi}_{\theta_j+l}, l\ge 1\}$. 
By the total probability law, the probability on the right hand side of \eqref{S.theta.1-} 
may be decomposed as follows:
\begin{eqnarray*}
\lefteqn{\sum_{n=1}^\infty \P\{\theta_j=\theta_{j-1}+n\le\tau(x),\ 
\widehat\xi_{\theta_{j-1}+n}>y_{k,j}\}}\\
&\le& \sum_{n=1}^\infty \P\{\theta_{j-1}+n\le\theta_j,\ \theta_{j-1}+n\le\tau(x),\  
\widehat\xi_{\theta_{j-1}+n}>y_{k,j}\}\\
&=& \sum_{n=1}^\infty \P\{\theta_{j-1}+n\le\theta_j,\ \theta_{j-1}+n\le\tau(x)\}  
\P\{\widehat\xi_1>y_{k,j}\},
\end{eqnarray*}
again due to the independence of the future property. 
Hence, owing to $\widehat\xi_j=\xi_j-a-\delta<\xi_j$,
\begin{eqnarray*}
\lefteqn{\P\{\theta_j\le\tau(x),\ \widehat\xi_{\theta_j}>y_{k,j}\}}\\
&\le& \P\{\xi_1>y_{k,j}\} \sum_{n=1}^\infty \P\{\theta_{j-1}+n\le\theta_j,\ 
\theta_{j-1}+n\le\tau(x)\} \\
&=& \overline F(y_{k,j}) E_j,
\end{eqnarray*}
where
\begin{eqnarray*}
E_j &:=& \E \sum_{n=\theta_{j-1}+1}^{\theta_j} \I\{\tau(x)\ge n\}.
\end{eqnarray*}
Then it follows from \eqref{S.theta.1-} that, for $j\le k$,
\begin{eqnarray*}
\P\{\theta_k\le\tau(x),\ \theta_{k+1}=\infty,\ 
\widehat S_{\theta_j}-\widehat S_{\theta_{j-1}}>y_{k,j}\}
&\le& \overline F(y_{k,j}) p^{k-j} E_j,
\end{eqnarray*}
which being substituted into \eqref{upper.tau.crp.2.sum} yields
the following upper bound
\begin{eqnarray}\label{upper.tau.crp.2}
\P\{\widehat M_{\theta_k}>y,\ \theta_k\le\tau,\ \theta_{k+1}=\infty\}
&\le&  \sum_{j=1}^k \overline F(y_{k,j}) p^{k-j} E_j.
\end{eqnarray}
Since the events $\{\theta_k<\infty,\theta_{k+1}=\infty\}$, $k\ge 1$, 
form a partition of the event $\{\theta_1<\infty\}$, 
we take $y=\gamma x$, for some $\gamma >0$, and obtain  
\begin{eqnarray*}
\P\{\widehat M_{\tau(x)}>\gamma x\} &=& 
\P\{\widehat M_{\theta_k}>\gamma x,\ \theta_k\le\tau,\ \theta_{k+1}=\infty
\mbox{ for some }k\ge 1\}\\
&=& \sum_{k=1}^\infty 
\P\{\widehat M_{\theta_k}>\gamma x,\ \theta_k\le\tau,\ \theta_{k+1}=\infty\}.
\end{eqnarray*}
Therefore, it follows from \eqref{upper.tau.crp.2} that
\begin{eqnarray*}
\P\{\widehat M_{\tau(x)}>\gamma x\} 
&\le& \sum_{k=1}^\infty \sum_{j=1}^k \overline F(y_{k,j}) p^{k-j} E_j\\
&=& \sum_{j=1}^\infty E_j \sum_{k=0}^\infty \overline F(y_{k+j,j}) p^k.
\end{eqnarray*}
We set $y_{k,j}=\gamma x(1-\alpha)\alpha^{k-j}$ with some $\alpha\in(0,1)$, then
\begin{eqnarray}\label{M.tau.upper.1}
\P\{\widehat M_{\tau(x)}>\gamma x\} 
&\le& \E\tau(x) \sum_{k=0}^\infty \overline F(\gamma x(1-\alpha)\alpha^k) p^k.
\end{eqnarray}

Due to the condition \eqref{F.tau.cond},
\begin{eqnarray}\label{F.tau.cond.k}
\overline F(\gamma x(1-\alpha)\alpha^k)
&=& o\biggl(\frac{\P\{\tau>\gamma x(1-\alpha)\alpha^k\}}
{\E(\tau\wedge\gamma x(1-\alpha)\alpha^k)}\biggr)
\end{eqnarray}
as $x\to\infty$ uniformly for all $k$ such that $x\alpha^k\to\infty$.
Let $\beta\in(p,1)$. It follows from the intermediate regular variation of $\tau$    
that there exist an $\alpha<1$ and a $y_0$ such that 
\begin{eqnarray*}
\P\{\tau>\alpha y\} &\le& \beta^{-1}\P\{\tau>y\}\quad\mbox{for all }y>y_0.
\end{eqnarray*}
Then there exists a constant $c<\infty$ such that
\begin{eqnarray*}
\P\{\tau>\alpha^k y\} &\le& c\beta^{-k}\P\{\tau>y\}
\quad\mbox{for all }y>0\mbox{ and }k\ge 1,
\end{eqnarray*}
Also,
\begin{eqnarray*}
\E(\tau\wedge\alpha^k y)  &=& \int_0^{\alpha^k y} \P\{\tau>u\}du
\ =\ \alpha^k\int_0^y \P\{\tau>\alpha^k u\}du\\
&\ge& \alpha^k \int_0^y \P\{\tau>u\}du\\ 
&=& \alpha^k \E(\tau\wedge y)
\quad\mbox{for all }y>0\mbox{ and }k\ge 1.
\end{eqnarray*}
 Substituting these bounds into \eqref{F.tau.cond.k}, we conclude that,
 as $x\to\infty$, uniformly for all $k$ such that $x\alpha^k\to\infty$,
 \begin{eqnarray}\label{F.tau.cond.k.2}
\overline F(\gamma x(1-\alpha)\alpha^k)
&=& o\biggl(\frac{\beta^{-k}\alpha^{-k}\P\{\tau>\gamma x(1-\alpha)\}}
{\E(\tau\wedge\gamma x(1-\alpha))}\biggr)\nonumber\\
&=& o\biggl(\frac{(\alpha\beta)^{-k}\P\{\tau>x\}}{\E\tau(x)}\biggr),
\end{eqnarray}
owing to the intermediate regular variation of $\tau$.
Hence, \eqref{M.tau.upper.1} yields that
\begin{eqnarray*}
\P\{\widehat M_{\tau(x)}>\gamma x\} 
&=& o(\P\{\tau>x\})\quad\mbox{as }x\to\infty,
\end{eqnarray*}
provided $\alpha\beta>p$.
This implies \eqref{Mtau.upper} and hence completes the proof 
of the upper tail bound \eqref{S.tau.upper.bound} for $M_\tau$.

Now let us turn to a matching lower bound. We will prove that
\begin{eqnarray}\label{S.tau.lower.bound}
\P\{S_\tau>x\} &\ge& (1+o(1))\P\{\tau>x/a\}\quad\mbox{as }x\to\infty.
\end{eqnarray}
For $\varepsilon>\delta>0$, we have
\begin{eqnarray*}
\P\{S_\tau>x\} &\ge& \P\Bigl\{S_\tau>x,\ \tau>\frac{x}{a-\varepsilon}\Bigr\}\\
&=& \P\Bigl\{\tau>\frac{x}{a-\varepsilon}\Bigr\}
-\P\Bigl\{S_\tau\le x,\ \tau>\frac{x}{a-\varepsilon}\Bigr\},
\end{eqnarray*}
so it suffices to show that
\begin{eqnarray*}
\P\Bigl\{S_\tau\le x,\ \tau>\frac{x}{a-\varepsilon}\Bigr\}
&=& o\Bigl(\P\Bigl\{\tau>\frac{x}{a-\varepsilon}\Bigr\}\Bigr)\quad\mbox{as }x\to\infty,
\end{eqnarray*}
or, equivalently (owing to the intermediate regular variation of $\tau$
satisfying \eqref{i.r.v.gamma}), that, for some $\gamma>0$,
\begin{eqnarray}\label{lower.suff}
\P\Bigl\{S_\tau\le x,\ \tau>\frac{x}{a-\varepsilon}\Bigr\}
&=& o(\P\{\tau>\gamma x\})\quad\mbox{as }x\to\infty.
\end{eqnarray}
In order to obtain \eqref{lower.suff}, we bound this probability as follows:
\begin{eqnarray}\label{Mtau.lower}
 \P\Bigl\{S_\tau\le x,\ \tau>\frac{x}{a-\varepsilon}\Bigr\} &=&
\P\Bigl\{S_\tau-(a-\delta)\tau\le x-(a-\delta)\tau,\ \tau>\frac{x}{a-\varepsilon}\Bigr\}\nonumber\\
&\le& \P\Bigl\{S_\tau-(a-\delta)\tau\le -x\frac{\varepsilon-\delta}{a-\varepsilon},\ 
\tau>\frac{x}{a-\varepsilon}\Bigr\}\nonumber\\
&=& \P\Bigl\{\widecheck S_\tau\le -x\frac{\varepsilon-\delta}{a-\varepsilon},\ 
\tau>\frac{x}{a-\varepsilon}\Bigr\} =: P_2(x),
\end{eqnarray}
where $\widecheck S_k:=S_k-(a-\delta)k = \sum_{j=1}^n \widecheck \xi_j$ is a random walk 
with positive drift ${\mathbb E} \widecheck \xi_j = {\mathbb E} (\xi_j - a +\delta)= \delta$.

To bound the probability $P_2(x)$, we introduce recursively a sequence
$\{\theta_k\}_{k\ge 0}$ of stopping times by letting $\theta_0=0$ and, for $k\ge 0$,
\begin{align*}
\theta_{k+1}
 := \min \{n>\theta_k: \ \widecheck S_n-\widecheck S_{\theta_k} > (n-\theta_k)\delta/2\},
\end{align*}
so that $\widecheck S_{\theta_k}> \theta_k\delta/2$.
All these $\theta$'s are proper random variables due to the positive drift 
$\delta$ of the random walk $\widecheck S_n$.
Moreover, $\theta_k-\theta_{k-1}$, $k\ge 1$, 
are iid random variables with finite mean value, 
\begin{eqnarray}\label{theta.crp}
\E\theta_1 &<& \infty.
\end{eqnarray}
In addition, the distribution $G$ defined by its tail
\begin{eqnarray*}
\overline G(x) &:=& \frac{\P\{\tau>x\}}{\E(\tau\wedge x)},
\end{eqnarray*}
is intermediate regularly varying at infinity as the distribution of $\tau$ is so.
Then, under the condition \eqref{F.tau.cond}, 
it follows from Theorem \ref{thm:cycle.asy.o} below that
\begin{eqnarray}\label{tau.xi.1}
\P\{\theta_1>n\}  &=& o(\overline G(n))\quad\mbox{as }n\to\infty.
\end{eqnarray}
Introduce the minima over disjoint time intervals by
\begin{align*}
L_k := \min_{ \theta_k\le n\le\theta_{k+1}}
\widecheck S_n-\widecheck S_{\theta_k}\ \le\ 0,\quad k\ge 0,
\end{align*}
and notice that
\begin{eqnarray}\label{L.crp}
\{L_k,\ k\ge 0\}\quad\mbox{are iid proper random variables and}\quad \E L_0>-\infty,
\end{eqnarray}
thanks to the inequality 
$|L_0| \le \xi_1^- +\ldots+\xi_{\theta_1}^-$  and Wald's identity. Then, 
\begin{eqnarray*}
P_2(x) 
&=& \sum_{k=0}^\infty \P \Bigl\{\theta_k<\tau\le \theta_{k+1},\
\widecheck S_\tau \le -x\frac{\varepsilon-\delta}{a-\varepsilon},\ 
\tau>\frac{x}{a-\varepsilon}\Bigr\}\\
&\le& \sum_{k=0}^\infty\P\Bigl\{\theta_k<\tau\le \theta_{k+1},\
\widecheck S_\tau-\widecheck S_{\theta_k}\le -x\frac{\varepsilon-\delta}{a}-\delta\theta_k/2,\
\tau>\frac{x}{a-\varepsilon}\Bigr\}\\
&\le& \sum_{k=0}^\infty \P\Bigl\{\theta_k<\tau\le \theta_{k+1},\
\widecheck S_\tau-\widecheck S_{\theta_k} \le -x\frac{\varepsilon-\delta}{a}-\delta k/2,\ 
\tau>\frac{x}{a-\varepsilon}\Bigr\},
\end{eqnarray*}
due to $\theta_k\ge k$. Therefore,
\begin{eqnarray}\label{sigma1.sigma2}
P_2(x)
&\le& \sum_{k=0}^\infty \P\Bigl\{\theta_k<\tau\le \theta_{k+1},\
L_k \le -x\frac{\varepsilon-\delta}{a}-\delta k/2,\ 
\tau>\frac{x}{a-\varepsilon}\Bigr\}\nonumber\\
&=& \sum_{k=0}^\infty \P\Bigl\{\gamma x<\theta_k<\tau\le \theta_{k+1},\
L_k \le -x\frac{\varepsilon-\delta}{a}-\delta k/2,\ 
\tau>\frac{x}{a-\varepsilon}\Bigr\}\nonumber\\
&& +\sum_{k=0}^\infty \P\Bigl\{\theta_k\le\gamma x,\ \theta_k<\tau\le \theta_{k+1},\
L_k \le -x\frac{\varepsilon-\delta}{a}-\delta k/2,\ 
\tau>\frac{x}{a-\varepsilon}\Bigr\}\nonumber\\
&=:& \Sigma_1(x)+\Sigma_2(x),
\end{eqnarray}
for any $\gamma>0$. Firstly,
\begin{eqnarray}\label{sigma1}
\Sigma_1(x) &\le& \sum_{k=0}^\infty \P\Bigl\{\gamma x<\theta_k<\tau,\
L_k \le -x\frac{\varepsilon-\delta}{a}-\delta k/2\Bigr\}.
\end{eqnarray}
Since $\tau$ is independent of the future increments, the event 
$$
\{\gamma x<\theta_k<\tau\}\ =\ \{\gamma x<\theta_k\}\cap\overline{\{\tau\le\theta_k\}}
$$
does not depend on the random variable $L_k$.

Hence, we conclude that the sum on the right-hand side of \eqref{sigma1} is equal to 
\begin{eqnarray*}
\sum_{k=0}^\infty \P\{\gamma x<\theta_k<\tau\}
\P\Bigl\{L_k \le -x\frac{\varepsilon-\delta}{a}-\delta k/2\Bigr\},
\end{eqnarray*}
which may be bounded above by
\begin{eqnarray*}
\P\{\tau>\gamma x\}\sum_{k=0}^\infty
\P\Bigl\{L_0 \le -x\frac{\varepsilon-\delta}{a}-\delta k/2\Bigr\},
\end{eqnarray*}
owing to \eqref{L.crp}.
Here the sum on the right hand side 
can be bounded by
\begin{eqnarray*}
\sum_{k=0}^\infty \frac{2}{\delta}\int_{\delta k/2}^{\delta(k+1)/2} 
\P\Bigl\{L_0 \le -x\frac{\varepsilon-\delta}{a}-y\Bigr\}dy
&=& \frac{2}{\delta} \int_0^\infty
\P\Bigl\{L_0 \le -x\frac{\varepsilon-\delta}{a}-y\Bigr\}dy\\
&\to& 0\quad\mbox{as }x\to\infty,
\end{eqnarray*}
since $\E L_0>-\infty$. Hence, for all $\gamma>0$,
\begin{eqnarray}\label{bound.for.sigma1}
\Sigma_1(x) &=& o(\P\{\tau>\gamma x\})\quad\mbox{as }x\to\infty.
\end{eqnarray}
Secondly, for $\gamma<1/a$,
\begin{eqnarray*}
\Sigma_2(x) &\le& \sum_{k=0}^\infty \P\bigl\{\theta_k\le \gamma x,\ \theta_k<\tau,\
\theta_{k+1}-\theta_k>x(1/a-\gamma)\bigr\}\\
&=& \sum_{k=0}^\infty \P\{\theta_k\le \gamma x,\ \theta_k<\tau\}
\P\{\theta_{k+1}-\theta_k>x(1/a-\gamma)\}\\
&=& \P\{\theta_1>x(1/a-\gamma)\}
\sum_{k=0}^\infty \P\{\theta_k\le \gamma x,\ \theta_k<\tau\},
\end{eqnarray*}
again by the independence of the future property of $\tau$ and by \eqref{L.crp}. 
It follows from \eqref{tau.xi.1} that
\begin{eqnarray}\label{tau.xi.2}
\P\{\theta_1>x(1/a-\gamma)\} 
&=& o\biggl(\frac{\P\{\tau>x(1/a-\gamma)\}}
{\E(\tau\wedge x(1/a-\gamma))}\biggr)
\quad\mbox{as }x\to\infty.
\end{eqnarray}
Therefore,
\begin{eqnarray*}
\Sigma_2(x) &=& o\biggl(\frac{\P\{\tau>x(1/a-\gamma)\}}
{\E(\tau\wedge x(1/a-\gamma))}\biggr)
\sum_{k=0}^\infty \P\{\theta_k\le (\tau\wedge\gamma x)\}\quad\mbox{as }x\to\infty.
\end{eqnarray*}
Since $\theta_k\ge k$,
\begin{eqnarray*}
\sum_{k=0}^\infty \P\{\theta_k\le (\tau\wedge\gamma x)\} &\le&
\sum_{k=0}^\infty \P\{\tau\wedge\gamma x\ge k\}\\
&\le& \E(\tau\wedge \gamma x)+1.
\end{eqnarray*}
Hence, similar to \eqref{F.tau.cond.k.2}, it follows from the last two bounds that
\begin{eqnarray}\label{bound.for.sigma2}
\Sigma_2(x) &=& o(\P\{\tau>x\})\quad\mbox{as }x\to\infty.
\end{eqnarray}
Substituting the upper bounds \eqref{bound.for.sigma1} and \eqref{bound.for.sigma2} 
into \eqref{sigma1.sigma2}, we conclude \eqref{lower.suff}
and hence obtain the lower bound \eqref{S.tau.lower.bound}.
\end{proof}

\section{Arbitrary dependence between $\{\xi_n\}$ and $\tau$}
\label{sec:arb}

In Theorem \ref{thm:asy}, we require $\tau$ to be independent of the future
of the sequence $\{\xi_n\}$. Now we show that this assumption can be omitted
at the cost of strengthening the conditions on both the right and left tails
of the distribution of $\xi$.
%

\begin{theorem}\label{thm:asy.sstar}
Assume that $a:=\E\xi>0$ and that the distribution of $\tau$ is intermediate regularly varying at infinity.
If
\begin{eqnarray}\label{F.tau.cond.sstar.upper}
x \P\{\xi>x\} &=& o(\P\{\tau>x\})\quad\mbox{as }x\to\infty,
\end{eqnarray}
then 
\begin{eqnarray}\label{F.tau.asy.gen.upper}
\P\{M_\tau>x\} &\le& (1+o(1))\P\{\tau>x/a\}\quad\mbox{as  } x\to\infty.
\end{eqnarray}
If
\begin{eqnarray}\label{F.tau.cond.sstar.lower}
\int_x^\infty \P\{\xi<-y\}dy &=& o(\P\{\tau>x\})\quad\mbox{as }x\to\infty,
\end{eqnarray}
then 
\begin{eqnarray}\label{F.tau.asy.gen.lower}
\P\{S_\tau>x\} &\ge& (1+o(1))\P\{\tau>x/a\}\quad\mbox{as  } x\to\infty.
\end{eqnarray}
If both conditions, \eqref{F.tau.cond.sstar.upper} and \eqref{F.tau.cond.sstar.lower}, hold true then 
\begin{eqnarray}\label{F.tau.asy.gen}
\P\{M_\tau>x\}\ \sim\ \P\{S_\tau>x\} &\sim& \P\{\tau>x/a\}\quad\mbox{as  } x\to\infty.
\end{eqnarray}
\end{theorem}

Let us remark that the conditions above are not symmetric for the upper and the lower tails for a good reason.
Firstly note that the condition
\begin{eqnarray*}
\int_x^\infty \P\{\xi>y\}dy &=& o(\P\{\tau>x\})
\end{eqnarray*}
would imply the condition \eqref{F.tau.cond.sstar.upper} in the case of intermediate regularly varying $\tau$,  because
\begin{eqnarray*}
\int_x^\infty \P\{\xi>y\}dy\ \ge\ \int_x^{2x} \P\{\xi>y\}dy &\ge& x\P\{\xi>2x\}.
\end{eqnarray*}
However, the converse implication does not hold, as the following example demonstrates:
in the case where $\P\{\xi>x\}\sim 1/x\log^2x$ as $x\to\infty$, 
\begin{eqnarray*}
\int_x^\infty \P\{\xi>y\}dy &\sim& \frac{1}{\log x} \ \gg\ \frac{1}{\log^2 x}
\ \sim\ x\P\{\xi>x\}\quad\mbox{as  } x\to\infty.
\end{eqnarray*}

In \cite[Theorem 3.1]{RS}, arbitrary dependence of $\{\xi_n\}$ and $\tau$ is considered
under the condition \eqref{F.tau.cond.sstar.upper}. 
In addition, it is assumed there that $\xi\ge 0$ and $\E\xi^r<\infty$ for some $r>1$;
the latter condition excludes distributions such that the one above.
Our result shows that the condition \eqref{F.tau.cond.sstar.upper} is indeed sufficient
for the upper bound without any further assumptions on $\xi$.

For the lower bound \eqref{F.tau.asy.gen.lower}, the stronger condition \eqref{F.tau.cond.sstar.lower} 
on the left tail of $\xi$ is needed due to the proving technique we follow;
we believe it is not only because of the proving technique but also otherwise the lower bound may fail.

\begin{proof}[Proof of Theorem \ref{thm:asy.sstar}.]
In the proof of the upper tail bound for $M_\tau$ in the last theorem, 
we do not make use of the independence of the future up to
the equation \eqref{Mtau.upper}. Therefore, it suffices to show the following tail bound 
for the running maximum $\widehat M_n$ of a negatively driven random walk $\widehat S_n$
with jumps $\widehat\xi_j=\xi_j - a-\delta$, $\E \widehat \xi_j =-\delta$,
 until time $x/(a+\varepsilon)$:
\begin{eqnarray}\label{wh.minfty.upper}
\P\Bigl\{\widehat M_{x/(a+\varepsilon)}>x\frac{\varepsilon-\delta}{a+\varepsilon}\Bigr\}
&=& o(\P\{\tau>x\})\quad\mbox{as }x\to\infty.
\end{eqnarray}
Thus, we may start the current proof from here. Without loss of generality,
\begin{eqnarray*}
\overline{\widehat F}(x) &\le& \frac{\P\{\tau>x\}}{x}\quad\mbox{for all }x\ge 1,
\end{eqnarray*}
where $\widehat F$ is the distribution of $\widehat\xi_1$. Consider a distribution $G$ with 
\begin{eqnarray*}
\overline G(x) &:=& \left\{
\begin{array}{ll}
\P\{\tau>x\}/x&\mbox{if }x\ge 1,\\
1&\mbox{if }x<1.
\end{array}
\right.
\end{eqnarray*}

From the condition $\overline{\widehat F}(x)=o(\overline G(x))$, we can identify an increasing concave 
function $g(x)\uparrow\infty$ such that $\overline{\widehat F}(x)\le\overline G(x)/g(x)$.
As a concave function, $g(x)$ is intermediate regularly varying at infinity.
Since $G$ is intermediate regular varying as well, the distribution $H$
defined via its tail as $\overline H(x):=\overline G(x)/g(x)$ is also intermediate
regularly varying. We have $\overline{\widehat F}(x)\le\overline H(x)$
for all $x>1$ and $\overline H(x)=o(\overline G(x))$ as $x\to\infty$.

By the main result in \cite{D}, there exists a regularly varying distribution $J$ with index $-1$ 
and finite expectation such that $\overline{\widehat F}(x)\le\overline J(x)$ for all $x$.
The distribution $V$, defined via its tail as $\overline V(x):=\min(\overline J(x),\overline H(x))$, 
is intermediate regularly varying, as both  $J$ and $H$ are, and its expectation is finite because $J$ has a finite expectation.
Thus, $V$ is a strong subexponential distribution by \cite[Theorem 3.29]{FKZ}.

Fix $\gamma \in (0,1)$ and $\widetilde x\in\R$, and consider an intermediate
regularly varying distribution $\widetilde F$ defined as follows:
\begin{eqnarray*}
\widetilde F(x)  &:=& 
\left\{
\begin{array}{ll}
(1-\gamma)\widehat F(x) &\mbox{if }x\le \widetilde x;\\
1-\gamma\overline V(x) &\mbox{if }x>\widetilde x.
\end{array}
\right.
\end{eqnarray*}
Since $\widetilde F$ has a finite mean, it is also strong subexponential.

Let $\widetilde\xi_n$, $n\ge 1$, be independent random variables with common distribution $\widetilde F$.
Let $\widetilde S_n=\widetilde\xi_1+\ldots+\widetilde\xi_n$.
By the construction, $\widetilde\xi_1\ge_{st}\widehat\xi_1$ and hence $\widetilde S_n\ge_{st}\widehat S_n$.
Therefore, $\widetilde M_n\ge_{st}\widehat M_n$, which means that 
\begin{eqnarray}\label{Mhat.Mtilde}
\P\{\widehat M_n>y\} &\le& \P\{\widetilde M_n>y\}\quad\mbox{for all }n\ge 1\mbox{ and }y.
\end{eqnarray}
Choose $\gamma>0$ so small and $\widetilde x\in\R$ so large that $\E\widetilde\xi<-\delta/2$.
Since the distribution $\widetilde F$ is strong subexponential,
we may apply Theorem 5.3 in \cite{FKZ} to derive that, as $y\to\infty$,
\begin{eqnarray*}
\P\{\widetilde M_n>y\} &\sim&
 \frac{1}{|\E\widetilde\xi|} \int_y^{y+n|\E\widetilde\xi|} \P\{\widetilde\xi>z\}dz\\
&=& \frac{\gamma}{|\E\widetilde\xi|} \int_y^{y+n|\E\widetilde\xi|} \overline V(z)dz\\
&\le& n\gamma \overline V(y),
\end{eqnarray*}
where the equivalence in the first line is uniform for all $n$. 
Hence, it follows from \eqref{Mhat.Mtilde} that, as $y\to\infty$,  uniformly for all $n\ge 1$,
\begin{eqnarray*}
\P\{\widehat M_n>y\} &\le& n(\gamma+o(1)) \overline V(y).
\end{eqnarray*}
which implies that
\begin{eqnarray*}
\P\Bigl\{\widehat M_{x/(a+\varepsilon)}>x\frac{\varepsilon-\delta}{a+\varepsilon}\Bigr\} 
&\le& \frac{x}{a+\varepsilon}(\gamma+o(1)) \overline V\Bigl(x\frac{\varepsilon-\delta}{a+\varepsilon}\Bigr)\\
&=& O(x\overline V(x))\quad\mbox{as }x\to\infty,
\end{eqnarray*}
since $V$ is intermediate regularly varying. Taking into account that 
$\overline V(x)\le\overline H(x)=o(\overline G(x))$ we conclude \eqref{wh.minfty.upper}
and hence \eqref{F.tau.asy.gen.upper}.

Similar to the upper bound, the proof of the lower tail bound in the last theorem 
applies up to the equation \eqref{Mtau.lower} 
whatever a dependence between $\{\xi_n\}$ and $\tau$  is. Therefore, 
it suffices to establish the following tail bound for the overall minimum $\widecheck M_\infty$
of the positively driven random walk $\widecheck S_n$ 
with jumps $\widecheck\xi_k=\xi-a+\delta$ having positive mean $\E\widecheck\xi_1=\delta$:
\begin{eqnarray}\label{wh.minfty.lower}
\P\{\widecheck M_\infty<-x\}
&=& o(\P\{\tau>x\})\quad\mbox{as }x\to\infty.
\end{eqnarray}
Thus, we may start the current proof from here. 

It follows from the ladder height structure of a random walk (see e.g. \cite[(5.20)]{FKZ}) 
that the distribution of the overall maximum possesses the following representation:
for any Borel set $B\subseteq(-\infty,0)$,
\begin{eqnarray*}
\P\{\widecheck M_\infty\in B\} &=& p\sum_{k=1}^\infty
\P\{\psi_-(1)+\ldots+\psi_-(k)\in B\}(1-p)^k,
\end{eqnarray*}
where $\psi_-(k)$ are independent copies of the first strictly descending ladder height
$\widecheck S_{\eta_-}$ conditioned to $\eta_-<\infty$; here $\eta_-$ is
the first strictly descending ladder epoch, $\eta_-:=\min\{k\ge1:\widecheck S_k<0\} \le\infty$
which is a defective random variable, so
\begin{eqnarray*}
p &:=& \P\{\widecheck M_\infty=0\}=\P\{\widecheck S_k\ge0\mbox{ for all }k\ge1\}
=\P\{\eta_-=\infty\}\ >\ 0.
\end{eqnarray*}
It follows from \cite[Section 5.5]{FKZ} that there exists a constant $c<\infty$ such that
\begin{eqnarray*}
\P\{\psi_-(1)<-x\} &\le& c\int_x^\infty \P\{\widecheck\xi<-y\}dy\quad\mbox{for all }x>0.
\end{eqnarray*}
Therefore, owing to the condition \eqref{F.tau.cond.sstar.upper}, for any $\gamma>0$,
there exists an $x_0$ such that
\begin{eqnarray*}
\P\{\psi_-(1)< -x\} &\le& \gamma \P\{\tau>x\}\quad\mbox{for all }x\ge x_0.
\end{eqnarray*}
We can choose $x_0$ so large that $\gamma \P\{\tau>x_0\}<1$.
Let us consider a random variable $\psi_-^{(\gamma)}$ with distribution
\begin{eqnarray*}
\P\{\psi_-^{(\gamma)}<-x\} &=& \left\{
\begin{array}{ll}
1 &\mbox{ for }x<x_0;\\
\gamma\P\{\tau>x\} & \mbox{ for }x\ge x_0.
\end{array}
\right.
\end{eqnarray*}
By the construction, $\psi_-\ge_{st}\psi_-^{(\gamma)}$, which implies that
\begin{eqnarray*}
\P\{\widecheck M_\infty<-x\} &\le& p\sum_{k=1}^\infty
\P\{\psi_-^{(\gamma)}(1)+\ldots+\psi_-^{(\gamma)}(k)<-x\}(1-p)^k,
\end{eqnarray*}
where $\psi_-^{(\gamma)}(n)$, $n\ge 1$, are independent copies of $\psi_-^{(\gamma)}$.
Since $\tau$ has an intermediate regularly varying distribution,
it is subexponential, see e.g. \cite[Section 3.5]{FKZ}. 
Hence, the random variable $\psi_-^{(\gamma)}$ is subexponential too.
Therefore, we may apply Theorem 3.37 from \cite{FKZ} to conclude that, as $x\to\infty$,
\begin{eqnarray*}
p\sum_{k=1}^\infty
\P\{\psi_-^{(\gamma)}(1)+\ldots+\psi_-^{(\gamma)}(k)<-x\}(1-p)^k
&\sim& \frac{1-p}{p}\P\{\psi_-^{(\gamma)}<-x\}\\
&=& \frac{1-p}{p}\gamma\P\{\tau>x\}\quad\mbox{for }x>x_0.
\end{eqnarray*}
Thus, for any fixed $\gamma>0$,
\begin{eqnarray*}
\limsup_{x\to\infty}\frac{\P\{\widecheck M_\infty<-x\}}{\P\{\tau>x\}} 
&\le& \frac{1-p}{p}\gamma,
\end{eqnarray*}
which in turn yields \eqref{wh.minfty.lower}. This completes the proof of the lower bound  
\eqref{F.tau.cond.sstar.lower}.
\end{proof}

\section{Tail asymptotics for the cycle length}
\label{sec:cycle}

Let $\theta=\min\{n\ge 1:S_n\le 0\}$.
In the next result, we derive the tail asymptotics for $\theta$
in the intermediate regularly varying case. A similar result
can be found in \cite{DSh} under the additional condition that
$\E|\xi|^r<\infty$ for some $r>1$; 
the proof there is based on the approach suggested in \cite{EH}.
See also \cite[Theorem 3.1]{FM}, where an analogous result is proved
for the busy period in the framework of a single-server queue.
The proof in \cite{FM} is based on subexponential
asymptotics for a negatively driven random walk observed until a stopping time.

\begin{theorem}\label{thm:cycle.asy}
Let $\E\xi=-a<0$ and let the distribution $F$ of $\xi$ 
be intermediate regularly varying at infinity. Then
\begin{eqnarray*}
\P\{\theta>n\} &\sim& \E\theta\overline F(na)\quad\mbox{as }n\to\infty.
\end{eqnarray*}
\end{theorem}

\begin{proof}
We also follow the approach suggested in \cite{EH}.
Due to Feller \cite[Section XVIII.4]{Feller},
the probability generating function $\varphi(s)=\E s^\theta$ of $\theta$, $|s|\le 1$, satisfies the equality
\begin{eqnarray*}
\log\frac{1-\varphi(s)}{1-s} &=& \sum_{n=1}^\infty s^n\frac{\P\{S_n>0\}}{n}.
\end{eqnarray*}
In particular,
\begin{eqnarray*}
\lambda\ :=\ \log\E\theta &=& \sum_{n=1}^\infty \frac{\P\{S_n>0\}}{n}.
\end{eqnarray*}
Hence, the probability generating function $\widehat\varphi(s)$ of the probability
mass function $q_n:=\P\{\theta>n\}/\E\theta$, $n\ge 0$, satisfies the equalities
\begin{eqnarray}\label{rep.for.tau}
\widehat\varphi(s) \ =\ \frac{1}{\E\theta}\frac{1-\varphi(s)}{1-s} 
&=& e^{\lambda\sum_{n=1}^\infty (s^n-1)p_n},
\end{eqnarray}
where $p_n:=\P\{S_n>0\}/n\lambda$ is a probability mass function.
Let $\eta$ be a random variable with probability mass function $p_n$
and let $\eta_n$, $n\ge 1$, be independent copies of $\eta$.
Then it follows from \eqref{rep.for.tau} that 
\begin{eqnarray}\label{Poisson.repr}
q_n &=& \sum_{k=1}^\infty \frac{\lambda^k}{k!}e^{-\lambda}
\P\{\eta_1+\ldots+\eta_k=n\}.
\end{eqnarray}

Fix an $\varepsilon\in(0,a)$ and consider iid random variables 
$\xi_{\varepsilon,n}:=\xi_n+a-\varepsilon$ with negative mean 
$\E\xi_{\varepsilon,n}=-\varepsilon$.
Set $S_{\varepsilon,n}:=\xi_{\varepsilon,1}+\ldots+\xi_{\varepsilon,n}=S_n+n(a-\varepsilon)$ 
and $M_{\varepsilon,n}:=\max_{k\le n}S_{\varepsilon,k}$. 
Since any intermediate regularly varying distribution
with finite mean is strong subexponential (see, e.g. \cite[Theorem 3.29]{FKZ}),
we can apply \cite[Theorem 5.3]{FKZ} to conclude that
\begin{eqnarray*}
\P\{M_{\varepsilon,n}>x\} &\sim& \frac{1}{\varepsilon}
\int_x^{x+n\varepsilon} \P\{\xi_{\varepsilon,1}>y\}dy
\quad\mbox{as }x\to\infty\mbox{ uniformly for all }n\ge1.
\end{eqnarray*}
Therefore,
\begin{eqnarray*}
\P\{M_{\varepsilon,n}>x\} &\le& n(1+o(1)) \P\{\xi_{\varepsilon,1}>x\}\\
&\sim& n\P\{\xi>x\}\quad\mbox{as }x\to\infty\mbox{ uniformly for all }n\ge1.
\end{eqnarray*}
Taking into account the inequality
$S_n=S_{\varepsilon,n}-n(a-\varepsilon) \le M_{\varepsilon,n}-n(a-\varepsilon)$, we obtain that
\begin{eqnarray}\label{Sn.upper}
\P\{S_n>0\} &\le& \P\{M_{\varepsilon,n}>n(a-\varepsilon)\}\nonumber\\
&\le& n(1+o(1))\P\{\xi>n(a-\varepsilon)\}\quad\mbox{as }n\to\infty.
\end{eqnarray}
On the other hand, the events
$$
\{\xi_i>n(a+\varepsilon),\xi_k\le na\mbox{ for all }k\le n,k\not=i\}
$$ 
are disjoint for different $i\le n$, hence
\begin{eqnarray}\label{Sn.lower}
\lefteqn{\P\{S_n>0\}}\nonumber\\
&\ge& n\P\{\xi_n>n(a+\varepsilon),S_{n-1}>n(-a-\varepsilon),
\xi_k\le na\mbox{ for all }k\le n-1\}\nonumber\\
&=& n\P\{\xi_n>n(a+\varepsilon)\}
\P\{S_{n-1}>n(-a-\varepsilon),\xi_k\le na\mbox{ for all }k\le n-1\}\nonumber\\
&\sim& n\P\{\xi>n(a+\varepsilon)\}\quad\mbox{as }n\to\infty,
\end{eqnarray}
because
\begin{eqnarray*}
\lefteqn{\P\{S_{n-1}>n(-a-\varepsilon),\xi_k\le na\mbox{ for all }k\le n-1\}}\\
&\ge& 1-\P\{S_{n-1}\le n(-a-\varepsilon)\}-(n-1)\P\{\xi_1>na\}
\ \to\ 1\quad\mbox{as }n\to\infty,
\end{eqnarray*}
by the law of large numbers and the condition $\E|\xi_1|<\infty$.

Letting $\varepsilon\downarrow 0$ in \eqref{Sn.upper} and \eqref{Sn.lower}
and taking into account the intermediate regular variation of $\xi$, we conclude that
\begin{eqnarray*}
\P\{S_n>0\} &\sim& n\P\{\xi>na\}\quad\mbox{as }n\to\infty.
\end{eqnarray*}
This yields that the random variable $\eta$ satisfies the following asymptotic relation
\begin{eqnarray*}
\P\{\eta_1=n\} &\sim& \P\{\xi>na\}/\lambda\quad\mbox{as }n\to\infty,
\end{eqnarray*}
which ensures that the distribution of $\eta_1$ is locally subexponential, 
see \cite[Theorem 4.14]{FKZ}. In turn, this allows us to apply Kesten's bound
(see \cite[Theorem 4.11]{FKZ}) and to conclude that
\begin{eqnarray*}
\sum_{k=1}^\infty \frac{\lambda^k}{k!}e^{-\lambda}\P\{\eta_1+\ldots+\eta_k=n\}
&\sim& \P\{\eta_1=n\}\sum_{k=1}^\infty k\frac{\lambda^k}{k!}e^{-\lambda}\\
&=& \P\{\eta_1=n\}\lambda\\
&\sim& \P\{\xi>na\}\quad\mbox{as }n\to\infty.
\end{eqnarray*}
Then it follows from the representation \eqref{Poisson.repr} that
\begin{eqnarray*}
\P\{\theta>n\} &=& \E\theta q_n\ \sim\ \E\theta\P\{\xi>na\}\quad\mbox{as }n\to\infty
\end{eqnarray*}
and the proof is complete.
\end{proof}

In the next result we provide an upper bound for the tail distribution of $\theta$ 
in the case where $\xi$ is not necessarily intermediate regularly varying. 
We only assume that the tail distribution of $\xi$
is dominated by an intermediate regularly varying distribution,
not necessarily with finite expectation which is beyond the scope of 
Theorem \ref{thm:cycle.asy}.

\begin{theorem}\label{thm:cycle.asy.o}
Let $\E\xi=-a<0$ and let 
\begin{eqnarray}\label{cond.upper.xi}
\overline F(x)\ :=\ \P\{\xi>x\} &=& o(\overline G(x))\quad\mbox{as }x\to\infty,
\end{eqnarray}
for some intermediate regularly varying distribution $G$. Then
\begin{eqnarray*}
\P\{\theta>n\} &=& o(\overline G(n))\quad\mbox{as }n\to\infty.
\end{eqnarray*}
\end{theorem}

\begin{proof}
We follow arguments similar to those in the proof of the upper bound in 
Theorem \ref{thm:asy.sstar}. 

Without loss of generality, we may assume that $\overline F(x)\le\overline G(x)$ for all $x$.
From the condition $\overline F(x)=o(\overline G(x))$, we can identify an increasing concave 
function $g(x)\uparrow\infty$ such that $\overline F(x)\le\overline G(x)/g(x)$.
As a concave function, $g(x)$ is intermediate regularly varying at infinity.
Since $G$ is intermediate regular varying too, the distribution $H$
defined via its tail $\overline H(x):=\overline G(x)/g(x)$ is intermediate
regularly varying as well. We have $\overline F(x)\le\overline H(x)$
for all $x$ and $\overline H(x)=o(\overline G(x))$ as $x\to\infty$.

By the main result in \cite{D}, there exists a regularly varying distribution $J$ with index $-1$ 
and finite expectation such that $\overline F(x)\le\overline J(x)$ for all $x$.
The distribution $V$ defined via its tail as $\overline V(x):=\min(\overline J(x),\overline H(x))$
is intermediate regularly varying as $J$ and $H$ are, and its expectation is finite as the expectation of $J$ is.

Fix $\varepsilon>0$ and $\widehat x\in\R$ and consider an intermediate
regularly varying distribution $\widehat F$ defined as follows:
\begin{eqnarray*}
\widehat F(x)  &:=& 
\left\{
\begin{array}{ll}
(1-\varepsilon) F(x) &\mbox{if }x\le \widehat x;\\
1-\varepsilon\overline V(x) &\mbox{if }x>\widehat x.
\end{array}
\right.
\end{eqnarray*}
Let $\widehat\xi_n$, $n\ge 1$, be independent random variables with common distribution $\widehat F$.
Let $\widehat S_n=\widehat\xi_1+\ldots+\widehat\xi_n$.
By the construction, $\widehat\xi_1\ge_{st}\xi_1$ and hence $\widehat S_n\ge_{st}S_n$.
Therefore, $\widehat\theta\ge_{st}\theta$ where $\widehat\theta=\min\{n\ge 1:\widehat S_n\le 0\}$,
which implies that
\begin{eqnarray}\label{theta.thetahat}
\P\{\theta>n\} &\le& \P\{\widehat\theta>n\}\quad\mbox{for all }n\ge 1.
\end{eqnarray}
Choose $\varepsilon>0$ so small and $\widehat x\in\R$ so large that $\E\widehat\xi<0$.
Then Theorem \ref{thm:cycle.asy} applies to $\widehat S_n$, and we derive that
\begin{eqnarray*}
\P\{\widehat\theta>n\} &\sim& \E\widehat\theta\overline{\widehat F}(n|\E\widehat\xi|)\quad\mbox{as }n\to\infty.
\end{eqnarray*}
Hence,
\begin{eqnarray*}
\P\{\widehat\theta>n\} &=& O(\overline V(n|\E\widehat\xi|))\\
 &=& O(\overline V(n))\quad\mbox{as }n\to\infty,
\end{eqnarray*}
since $V$ is intermediate regularly varying. Taking into account that 
$\overline V(x)\le\overline H(x)=o(\overline G(x))$ we conclude that
\begin{eqnarray*}
\P\{\widehat\theta>n\} &=& o(\overline G(n))\quad\mbox{as }n\to\infty,
\end{eqnarray*}
which together with \eqref{theta.thetahat} completes the proof,
\end{proof}

\end{document}